\documentclass[a4paper,11pt,leqno]{article}
\usepackage{amsmath,amssymb,amsfonts,amsthm,graphicx,fancyhdr,epic}
\usepackage[latin1]{inputenc}
\usepackage[T1]{fontenc}
\usepackage{rotating,datetime}

\newdateformat{mydate}{\THEYEAR/\THEMONTH/\THEDAY}

\newtheorem{teo}{Theorem}
\newtheorem{lem}[teo]{Lemma}
\newtheorem{cor}[teo]{Corollary}
\newtheorem{quest}[teo]{Question}

\newtheoremstyle{drem}
     {3pt}
     {3pt}
     {\rmfamily}
     {}
     {\itshape}
     {:}
     {.5em}
     {}

\theoremstyle{drem}
\newtheorem{rem}[teo]{Remark}
\newtheorem{exa}[teo]{Example}

\newcommand{\eg}[0]{\emph{e.g.} }
\newcommand{\ie}[0]{\emph{i.e.} }

\newcommand{\srl}[1]{\overline{#1}}

\newcommand{\eps}[0]{\epsilon}

\newcommand{\rr}[0]{\ensuremath{\mathbb{R}}}
\newcommand{\zz}[0]{\ensuremath{\mathbb{Z}}}

\newcommand{\tran}[0]{\ensuremath{\mathsf{T}}}
\newcommand{\tr}[0]{\ensuremath{\mathrm{Tr}}}
\newcommand{\Id}[0]{\ensuremath{\mathrm{Id}}}

\newcommand{\D}[0]{\ensuremath{\mathrm{D}}}
\renewcommand{\H}[0]{\ensuremath{\mathcal{H}}}
\newcommand{\HD}[0]{\ensuremath{\mathcal{H}\mathrm{D}}}
\newcommand{\BD}[0]{\ensuremath{\mathrm{BD}}}
\newcommand{\BHD}[0]{\ensuremath{\mathrm{B}\mathcal{H}\mathrm{D}}}

\newcommand{\supp}[0]{\ensuremath{\mathrm{supp}\,}}
\newcommand{\comp}[0]{\ensuremath{\mathsf{c}}}

\newcommand{\limm}[1]{\textrm{\raisebox{.5ex}{\mbox{$\underset{#1}{\lim}$}}} \:}

\title{Absence of non-constant harmonic functions with $\ell^p$-gradient in some semi-direct products}
\author{Antoine Gournay}
\date{\today}
\begin{document}

\maketitle

\begin{abstract}
To obtain groups with bounded harmonic functions (which are amenable), one of the most frequent way is to look at some semi-direct products (\eg lamplighter groups). The aim here is to show that many of these semi-direct products do not admit harmonic functions with gradient in $\ell^p$, for $p\in [1,\infty[$.
\end{abstract}

In \cite{moi-poiss} and \cite{moi-trans}, the author showed that many groups do not have non-constant harmonic functions with gradient in $\ell^p$ (for $p \in [1,\infty[$): \eg Liouville groups, lamplighters on $\zz^d$ with amenable lamp states, groups with infinitely many finite conjugacy classes, ... The aim of this short paper is to show that many semi-direct products (including lamplighter groups on bigger spaces) also have this property. This contrasts with the fact that all groups admit non-constant harmonic functions with gradient in $\ell^\infty$ (\ie Lipschitz) and that the groups under consideration have many non-constant bounded harmonic functions.

The graphs $\Gamma = (X,E)$ considered here will always be the Cayley graphs of finitely generated groups. The gradient of $f:X \to \rr$ is $\nabla f: E \to \rr$ defined by $\nabla f(x,y) = f(y) - f(x)$. The space of $p$-Dirichlet functions is $\D^p(\Gamma) = \{ f: X \to \rr \mid \nabla f \in \ell^p(E) \}$ and the space of harmonic functions is $\H(\Gamma) = \ker (\nabla^* \nabla)$. Harmonic functions with gradient in $\ell^p(E)$ are denoted $\HD^p(\Gamma) = \H(\Gamma) \cap \D^p(\Gamma)$, and $\BHD^p(\Gamma) = \HD^p(\Gamma \cap \ell^\infty(X)$ is the subspace of bounded such functions. 

Although these spaces depend \emph{a priori} on the generating set, an abuse of notation will be made by replacing the Cayley graph $\Gamma$ by the finitely generated group it represents: for a group $H$ and a finite generating set, $\D^p(H)$ is to be understood as the $\D^p$ space on the associated Cayley graph.
\begin{teo}\label{leteo}
Let $H$ be of growth at least polynomial of degree $d \geq 2p$ and, if $H$ is of superpolynomial growth, assume that only constant functions belong to $\BHD^q(H)$ for some $q>p$. Let $C$ be a group which is not finitely generated, assume $G = C \rtimes_\phi H$ is finitely generated and assume the hypothesis \eqref{hTC} hold. Then $\HD^p(G)$ contains only constant functions.
\end{teo}
The results actually holds for any graph quasi-isometric to a Cayley graph of $G$.

The hypothesis \eqref{hTC} (see subsection \S{}\ref{sscai}) means that elements in $C$ which are ``far away'' commute where ``far away'' can intuitively be thought of as that in order to make some specific subgroup which contain them intersect, one needs to apply a $\phi_h$ with $h \in H$ large. This hypothesis always holds for Abelian $C$, for wreath products and in some other interesting examples (see Example \ref{exsym}). The author is inclined to believe that other methods could avoid this hypothesis, however D.~Osin pointed out to the author it may not be removed.

Particular examples are lamplighter groups $L \wr H$, as long as $H$ has no bounded harmonic functions (and $d \geq 2p$). Using \cite[\S{}4]{moi-trans}, one can check that iterated wreath products (\eg $\zz \wr (\zz \wr \zz)$) have a trivial $\HD^p$. This gives a partial answers to a problem of Georgakopoulos \cite[Problem 3.1]{Agelos}, see \S{}\ref{s-exe}. See also \S{}\ref{s-comq} below for more details and \cite{moi-wreath} for more results on lamplighter groups.

Together with \cite[Theorem 1.4]{moi-poiss}, this results indicates that some extension operations should be avoided to construct groups with harmonic functions with gradient in $\ell^p$ out of groups which do not have them. In fact, using \cite[Theorem 1.2]{moi-poiss}, an open question of Gromov can be stated as: is it true that any finitely generated amenable group $G$ has only constant functions in $\HD^p(G)$ (for any $p \in ]2,\infty[$. For further questions and comments see \S{}\ref{s-comq} below.

Lemma \ref{tbdhnil-l} shows there are no harmonic functions with gradient in $c_0$ in groups of polynomial growth. It would be nice to have an amenable groups where this fails. Recall that for $p=2$, this result can be interpreted in terms of [reduced] $\ell^2$-cohomology in degree one (or first $\ell^2$-Betti numbers). For links between the current results and [reduced] $\ell^p$-cohomology [in degree $1$], the reader is directed to \cite{moi-poiss}.

{\it Acknowledgements}: I wish to thank an anonymous referee for many useful comments and important corrections, as well as D.~Osin for pointing out an example where \eqref{hTC} cannot be removed.

\section{Proof}

\newcommand{\lf}[0]{\lfloor}
\newcommand{\rf}[0]{\rfloor}
\newcommand{\fl}[1]{\lfloor #1 \rfloor}

\subsection{Boundary values}

The following lemma, taken from \cite{moi-poiss}, will come in handy. Let $\BD^p(\Gamma) = \D^p(\Gamma) \cap \ell^\infty(X)$ be the space of bounded functions in $\D^p$. 
\begin{lem}\label{tbdval}
Let $g \in \D^p(H)$ and $H$ be a group of growth at least polynomial of degree $d > 2p$. Let $P_H$ be the random walk operator on $H$ (for some finite generating set). Then $\tilde{g} = \lim_{n \to \infty} P^n_H g$ exists and there is a constant $K_1$ depending on the isoperimetric profile (in particular, possibly on $d$)
such that
\[
\|g - \tilde{g}\|_{\ell^\infty} \leq K_1 \|g\|_{\D^p}.
\]
Furthermore $\tilde{g} \in \D^q(H)$ for all $q \in \big[ \frac{pd}{d-2p}, \infty \big]$. If $g \in \BD^p(H)$ then $\tilde{g} \in \BD^q(H)$. 
\end{lem}
\begin{proof}
Let $g_n = P^n g$, then
\[
g - g_n = g - P^n g = \sum_{i=0}^{n-1} P^i g - P^{i+1} g = \sum_{i=0}^{n-1} P^i (I-P) g = \sum_{i=0}^{n-1} P^i (- \Delta g).
\]
But if $g \in D^p(H)$ then $\Delta g \in \ell^p(H)$. 

Let $p^{(i)} = P^i \delta_e$ where $\delta_e$ is the Dirac mass at the identity element of $H$. Note that the above expression reads $g -g_n = (-\Delta g) * \big(  \sum_{i=0}^{n-1} p^{(i)} \big)$. 

On the other hand if $H$ has polynomial growth of degree at least $d$ then $\|p^{(i)}\|_{\ell^r} \leq K n^{-d/2r'}$ where $r'$ is the H\"older conjugate of $r$. Indeed, use Varopoulos to have a bound on the $\ell^\infty$ norm: $\|p^{(n)}\|_{\ell^\infty} \leq K_2 n^{-d/2}$. The $\ell^1$-norm is always $1$. Interpolate to get the $\ell^r$ norm:
\[
\| p^{(n)}\|_{\ell^r}^r = \sum p^{(n)}(\gamma)^r \leq \|p^{(n)}\|^{r-1}_\infty \| p^{(n)}\|_1 \leq K_2^{r-1} n^{\tfrac{-d}{2} (r-1)}
\]
Recall that for $r,r_1,r_2 \in \rr_{\geq 1} \cup \{\infty\}$ satisfying $1+\tfrac{1}{r} = \tfrac{1}{r_1}+\tfrac{1}{r_2}$, Young's inequality (see \cite[Theorem 0.3.1]{Sogge}) gives $\|f*g \|_r \leq \|f \|_{r_1} \|g\|_{r_2}$. Applying this inequality to $g-g_n = p^{(n)} * (-\Delta g)$, one deduces convergence (of $g-g_n$ in $\ell^q$-norm) and 
\[
\|g - \tilde{g}\|_{\ell^q} \leq K_1(q) \|g\|_{\D^p}
\]
for $q \in \big[ \frac{pd}{d-2p} , \infty \big]$ and $K_1(q)$ is the product of the norm of $\nabla^*$ (from $\ell^p$ of the edges to $\ell^p$ of the vertices) and of Green's kernel (from $\ell^p(H) \to \ell^q(H)$). Thus
\[
\|g - \tilde{g}\|_{\ell^\infty} \leq K_1 \|g\|_{\D^p},
\]
where $K_1=K_1(\infty)$. For the last assertion, note that $\tilde{g}$ is harmonic bounded (given $g$ is bounded) and has gradient in $\ell^q$ (being the sum of a function in $\D^p(H)$ and a function in $\ell^q(H)$).
\end{proof}

\subsection{Some slicing and a reduction}

Take $G = C \rtimes_\phi H$. Assume $H$ is finitely generated (by $S_H$) and there is a finite set $S_C$ such that $\phi_H(S_C) := \{\phi_z(c) \mid z \in H, c \in S_C \}$ generates $C$. For $G$, consider the generating set $\{e_C\} \times S_H \cup S_C \times \{e_H\}$. This will turn out to be unimportant but makes the proof much simpler.

Recall $c_0$ is the closure of finitely supported functions in $\ell^\infty$: for some countable set $Y$, 
\[
c_0(Y) := \{ f:X \to \rr \mid \forall \eps>0, \exists F \subset Y \text{ finite such that } \|f\|_{\ell^\infty(X \setminus F} < \eps\}. 
\]
For $z \in H$ let $|z|$ be the word length of $z$ (for $S_H$). Let $\phi_F(S_C) := \{\phi_z(c) \mid z \in F, c \in S_C \}$. For $c \in C$, let $|c|$ be the word length for [the infinite alphabet] $\phi_H(S_C)$. Let 
\[
\lfloor c \rfloor := \min \{r \in \zz_{\geq 0} \mid c \text{ belongs to the subgroup generated by } \phi_{S_H^r}(S_C)  \}.
\]
Let $\supp c$ be union of the sets $F \subset H$ such that $c$ belongs to the group generated by $\phi_F(S_C)$, $F$ is minimal with respect to inclusion and $\max_{f \in F} |f| \leq \lfloor c \rfloor$. 

\begin{lem}\label{tepshl}
Let $f \in \D^p(\Gamma)$. There exists a sequence $\eps'_n$ of positive real numbers tending to $0$ 
\[
\forall c \in C, \forall s \in S_C \quad  |f(c,z) - f(c \cdot \phi_z(s), z)| < \eps'_{|z|}.  
\]
\end{lem}
\begin{proof} Writing the terms in $\nabla f$ gives
\[
\|f\|_{\D^p(G)}^p = \sum_{z \in Z} \sum_{c \in C} \bigg( \sum_{ s \in S_C} |f(c \cdot \phi_z(s),z)-f(c,z)|^p + \sum_{ s \in S_H} |f(c ,z \cdot h)-f(c,z)|^p \bigg)
\]
This implies that $\sum_{c \in C, s \in S_C} |f(c \cdot \phi_z(s),z)-f(c,z)|^p$ tends to $0$ in $z$. Since $\ell^p \subset c_0$, one has that $|f(c \cdot \phi_z(s),z)-f(c,z)|$ tends to  $0$ uniformly in $c$ and $s \in S_C$.
\end{proof}
Similarly:
\begin{lem}\label{tepscl}
Let $f \in \D^p(\Gamma)$. There exists a sequence $\eps_n$ of positive real numbers tending to $0$ so that 
\[
\|f(c,\cdot)\|_{\D^p(H)} \leq \eps_{\fl{c}}.
\]
\end{lem}
\begin{proof}
$\forall c \in C$, $f(c,\cdot) \in \D^p(H)$, and $\|f\|_{\D^p(G)}^p \geq \sum_{c \in C} \|f(c,\cdot)\|_{\D^p(H)}^p$. Again, the terms in this sum tend to $0$ (formally, we use again that $\ell^p \subset c_0$).
\end{proof}

The following lemma is probably well-known. 
\begin{lem}\label{tbdhnil-l}
Assume $H$ has polynomial growth. Then there are no harmonic functions with gradient in $c_0$.
\end{lem}
\begin{proof}
It is known by the works of Colding \& Minicozzi \cite{CM} (see also Kleiner \cite[Theorem 1.4]{Kleiner}) that groups of polynomial growth have a finite dimensional space of harmonic functions with gradient in $\ell^\infty$. Recall that $\lambda_\gamma f(x) := f(\gamma x)$. Since left-multiplication is an isometry of the Cayley graph, $\lambda_\gamma f$ is harmonic if and only if $f$ is. Furthermore, their gradients are the same up to this shift. Given $f$ non-constant with gradient in $c_0$, the aim is to show that the $\lambda_\gamma f$ span a vector space of arbitrarily large dimension. 

To do this, note that there is some edge $e \in E$ so that, up to multiplying $f$ by a constant, $\nabla f(e) =1$ (this is possible since $f$ is non-constant). For any $\eps$, let $\gamma_1$ be so that $|\nabla (\lambda_{\gamma_1} f) (e)| < \eps/2$, $|\nabla f( \gamma_1^{-1} e)| < \eps/2$. This is possible since $\nabla f \in c_0(E)$. Pick $\gamma_2$ so that $\lambda_{\gamma_2} f$ has gradient $<\eps/4$ at $e$ and $\gamma_1^{-1}e$ while $f$ and $\lambda_{\gamma_1} f$ have gradient $< \eps/4$ at $\gamma_2^{-1}e$. Continue this similarly to get a sequence $\gamma_i$ with ``errors'' $\eps/2^i$. Restricting to the edges $\gamma_i^{-1} e$ one sees vectors of the form:
\[
\begin{array}{ccccc}
 1 & \eps/2 & \eps/4 & \eps/8 & \ldots \\
 \eps/2 & 1 & \eps/4 & \eps/8 & \ldots \\
 \eps/4 & \eps/4 & 1 & \eps/8 & \ldots \\
 \eps/8 & \eps/8 & \eps/8 & 1 & \ldots \\
 \vdots & \vdots & \vdots & \vdots & \ddots
\end{array}
\]
Now if $L : \rr^N \to \rr^N$ is a linear map such that $\| L \vec{e}_i - \vec{e}_i \|_{\ell^2(N)} \leq \eps$ for $\{\vec{e}_i\}_{1 \leq i \leq N}$ the usual basis of $\rr^N$. Then a standard exercises shows $\dim \ker L \leq \eps^2 N$. 

Indeed, let $\{\vec{v}_i\}_{1 \leq i \leq N}$ be an orthogonal basis of $\rr^N$, and $M : \rr^N \to \rr^N$ be a linear map such that $M \vec{v}_i = \vec{v}_i$ for $1 \leq i \leq k$ then $k \leq \| M \|_{\ell^2(N^2)}^2$. This is because the $\ell^2(N^2)$ norm for these matrices can also be expressed by $\tr M^\tran M$ and is consequently independent of the choice of orthogonal basis. As $M\vec{v}_i = \vec{v}_i$ for $1 \leq i \leq k$, a simple computation yields $k \leq \tr M^t M = \|M\|_{\ell^2(N^2)}^2$.

Let $\dim \ker L = k$ and $M = L -\Id$. Since there is an orthogonal basis of $\rr^N$ such that the first $k$ elements actually form a basis of $\ker L$, this implies $\|M\|_{\ell^2(N^2)}^2 \geq k$. On the other hand, $\|M\|_{\ell^2(N^2)}^2  = \sum_i \|M \vec{e}_i\|_{\ell^2(N)}^2 \leq N \eps^2$. It follows that  $k \leq \eps^2 N$ as claimed.

This means that the space spanned by $\{ \lambda_{\gamma_i} f \}_{i=1}^{N}$ is of dimension at least $(1-\eps^2) N$. In particular there is an infinite dimensional space of Lipschitz harmonic functions. This implies the group is not of polynomial growth.
\end{proof}
The preceding lemma does not \emph{a priori} exclude the existence of Lipschitz harmonic functions with sublinear growth (this is a corollary of \cite[Theorem 1.3]{MPTY}).
\begin{rem}\label{rbhd}
Before moving on, it is necessary to note that groups of polynomial growth are Liouville, \ie they have no non-constant bounded harmonic functions. Thus for such groups $H$, $\BHD^q(H) \simeq \rr$ for any $q \in [1,\infty]$. In fact, if $q<\infty$, then $\HD^q(H)$ contains only constants, by Lemma \ref{tbdhnil-l}.

If $H$ has growth at least polynomial of degree $d >2p$ and $q \in \big] \frac{dp}{d-2p}, \infty \big]$, \cite[Theorem 1.2]{moi-poiss} shows that $\BHD^q(H) \simeq \rr$ implies $\HD^p(H) \simeq \rr$. In particular, if $H$ has superpolynomial growth, $q>p$ and $\BHD^q(H) \simeq \rr$ implies $\HD^p(H) \simeq \rr$.
\end{rem}

\subsection{Constant at infinity}\label{sscai}

Let us state the hypothesis \eqref{hTC} that is required on $\phi$ and $C$. It states that for any finite $F \subset H$, there are infinitely many $z$ so that there exists a $s$ satisfying $\phi_z(s)$ is not in the subgroup generated by $\phi_F(S_C)$ and 
\[\tag{TC} \label{hTC}
\text{for all } w \in F \text{ and } s' \in  S_C, \quad [\phi_z(s),\phi_w(s')] = 1. 
\]
The easiest case where this hypothesis hold is when $C$ is Abelian. In wreath products, this is also verified since $C$ is a direct product. For another common example see Example \ref{exsym}.

Assume $H$ is either of polynomial growth or has superpolynomial growth and $\BHD^q(H) \simeq \rr$ for some $q>p$. For $f \in \D^p(G)$ and $c \in C$, let $\tilde{f}(c,\cdot) = \limm{n \to \infty} P^n_H f(c,\cdot)$, where $P^n_H$ is the random walk operator restricted to $H$. Thus, by Lemma \ref{tepscl}, Lemma \ref{tbdval} and Remark \ref{rbhd}, $\| f(c,\cdot) - \text{cst}_c\|_{\ell^\infty(H)} \leq K_1 \eps_{\fl{c}}$, where $\text{cst}_c$ is the constant function $\tilde{f}(c,\cdot)$.

Given a set $F \subset H$, let $\srl{F} = \{ z \in H \mid \phi_z(S_C) \subset \phi_F(S_C) \}$. If $C$ is not finitely generated, then $H \setminus \srl{F}$ is infinite for any finite set $F$.
\begin{teo}\label{tcstinft}
Let $H$ be as above (\ie either of polynomial growth $d>2p$ or of superpolynomial growth and $\BHD^q(H) \simeq \rr$ for some $q>p$). Assume $C$ is not finitely generated.
Let $f \in \D^p(\Gamma)$, and define $\bar{f}: C \to \rr$ by $\bar{f}(c)$ is equal to the constant of the constant function $ \tilde{f}(c, \cdot)$. Then 
\[
|\bar{f}(c_1) - \bar{f}(c_2)| \leq K_1 ( \eps_{\fl{c_1}} + \eps_{\fl{c_2}} ).
\]
In particular, $\limm{\fl{c} \to \infty} \bar{f}(c)$ exists.
\end{teo}
\begin{proof}
We need to show that the constants $\text{cst}_{c_i}$ corresponding to $c_1$ and $c_2\in C$ are close. Let $|c_2^{-1}c_1|$ be the distance from $c_1$ to $c_2$ (in the infinitely generated Cayley graph of $C$ for the generating set $\phi_H(S_C)$). 

Fix some $\eps >0$ and assume for simplicity that $K_1 \geq 1$.

Let $z$ be so that $\nu := \max( \eps'_{|z|} , \eps_{|z|}) < \eps / K_1 (3|c_2^{-1}c_1| + 5)$ and $z$ lies outside $\srl{\Sigma}$ where $\Sigma = \supp c_1 \cup \supp c_2$. Since, for some $s \in S_C$ and any $c$ with $\srl{\supp c} \subset \srl{\Sigma}$, $\fl{c \cdot \phi_z(s)} \geq |z|$, by Lemmas \ref{tepscl} and \ref{tbdval},
\[
|f(c \cdot \phi_z(s),w) - \text{cst}_{c \cdot \phi_z(s)}| < K_1 \nu \quad \text{for any } w \in H \text{ and any } c \text{ with } \srl{\supp c} \subset \srl{\Sigma}.
\]
Also, by Lemma \ref{tepshl}, for any $c \in C$, $|f(c,z) - f(c \cdot \phi_z(s),z)| < \nu$. Thus, for any $s' \in S_C$, for any $w \in \srl{\Sigma}$ and any $c$ with $\srl{\supp c} \subset \srl{\Sigma}$,
\[
\begin{array}{l}
|\text{cst}_{c \cdot \phi_w(s') \cdot \phi_z(s)} - \text{cst}_{c \cdot \phi_z(s)}| \\
\qquad \leq |\text{cst}_{c \cdot \phi_w(s') \cdot \phi_z(s) } - f(c \cdot \phi_w(s') \cdot \phi_z(s) ,w)| \\
\qquad \qquad + |f(c \cdot \phi_w(s') \cdot \phi_z(s),w) - f(c \cdot \phi_z(s),w)| \\
\qquad \qquad + |f(c \cdot \phi_z(s),w) - \text{cst}_{c \cdot \phi_z(s)}| \\
\qquad \leq K_1 \nu + |f(c \cdot \phi_z(s) \cdot \phi_w(s'),w)  - f(c \cdot \phi_z(s),w)| +K_1 \nu.\\
\qquad \leq 2 K_1 \nu + \nu.\\
\end{array}
\]
where the above estimate as well as hypothesis \eqref{hTC} was used in the second inequality and Lemma \ref{tepscl} was used in the third inequality.
Hence, one has
\[
|\text{cst}_{c_1 \cdot \phi_z(s)} - \text{cst}_{c_2 \cdot \phi_z(s)}| \leq 3 K_1 |c_2^{-1}c_1| \nu.
\]
Finally,
\[
\begin{array}{rl}
|\text{cst}_{c_i \cdot \phi_z(s)} - \text{cst}_{c_i}|
& \leq |\text{cst}_{c_i \cdot \phi_z(s)} - f(c_i \cdot \phi_z(s),z)|  \\
& \qquad + |f(c_i \cdot \phi_z(s),z) - f(c_i,z)|  \\
& \qquad + |f(c_i,z) - \text{cst}_{c_i}| \\
& \leq (K_1 +1) \nu + K_1 \eps_{\fl{c_i}}.
\end{array}
\]
So
\[
\begin{array}{rl}
|\text{cst}_{c_1} - \text{cst}_{c_2}| 
& \leq  K_1 \big[ (3 |c_2^{-1}c_1| + 4) \nu +\eps_{\fl{c_1}} +\eps_{\fl{c_2}} \big]\\
& < K_1 \big( \eps + \eps_{\fl{c_1}} +\eps_{\fl{c_2}} \big).
\end{array}
\]
Since the above holds for any $\eps >0$, the conclusion follows.

To see the limit exists, note that the sequence is Cauchy.
\end{proof}

\begin{proof}[Proof of theorem \ref{leteo}]
First, the proof is done for the generating set build with $S_C$ and $S_H$. Let $B_n$ be a sequence of balls centred at the identity in $G$. Say a function has only one value at infinity if, up to changing $f$ by a constant function, $f( B_n^\comp) \subset ]-\sigma_n,\sigma_n[$ for some sequence of positive numbers $\sigma_n$ tending to $0$.

Take $f \in \HD^p(G)$. If $f$ takes only one value at infinity, then $f$ is constant (by the maximum principle). 

Thus, the theorem follows if we show any $f \in \D^p(G)$ has one value at infinity (it is not even required that $f$ be harmonic). Change $f$ by a constant so that the function $\bar{f}$ from Theorem \ref{tcstinft} tends to $0$ as $\fl{c}\to \infty$. This implies 
\[
|f(c,z)| = |f(c,z) - \text{cst}_c  + \text{cst}_c| \leq 3K_1 \eps_{\fl{c}},
\]
by bounding the first term as in the proof of Theorem \ref{tcstinft} and the second by the result of Theorem \ref{tcstinft}.
It remains to check that $f(c,z)$ also tends to $0$ as $|z| \to \infty$. Assume $z \notin S_H^{\fl{c}}$ (\ie $|z| > \fl{c}$), then, using the same bounds and Lemma \ref{tepshl}, 
\[
\begin{array}{rl}
|f(c,z)| 
&\leq |\text{cst}_{c \cdot \phi_z(s)}| + |\text{cst}_{c \cdot \phi_z(s)} - f(c \cdot \phi_z(s),z)| \\
& \qquad + |f(c \cdot \phi_z(s),z) - f(c,z) |  \\
& \leq 2K_1 \eps_{|z|} + K_1 \eps_{|z|} + \eps'_{|z|} \\
& \leq 3K_1\eps_{|z|} + \eps'_{|z|}.
\end{array}
\]
Thus $f$ has only one value at infinity.

If one considered another generating set, then a simple way is to do as follows. Note that $G$ is not virtually nilpotent, hence satisfies a $d$-dimensional isoperimetric profile for any $d$. By \cite[Theorem 1.4]{moi-poiss}, for any Cayley graph $\Gamma$ of $G$, one has: the reduced $\ell^p$-cohomology in degree one of $\Gamma$ is non-trivial for some $p \in [1,\infty[$ if and only if $\HD^q(\Gamma) \not\simeq \rr$ for some $q \in [1,\infty[$. Since the reduced $\ell^p$-cohomology in degree one is an invariant of quasi-isometry (in particular, of the choice of generating set) the result follows for other generating sets.
\end{proof}

\section{Some examples and questions}

\subsection{Examples}\label{s-exe}
Le us rewrite Theorem \ref{leteo}.
\begin{cor}
Let $G = C \rtimes_\phi H$, assume $C$ is not finitely generated but $G$ is and that hypothesis \eqref{hTC} holds. 
\begin{itemize}
\item If $H$ has polynomial growth of degree $d$, then, for all $p \in [1,d/2[$, $\HD^p(G)$ contains only constant functions.
\item If $H$ has intermediate growth, then, for all $p \in [1,\infty[$, $\HD^p(G)$ contains only constant functions.
\item If $H$ has exponential growth and $p \in [1,\infty[$ and let $1 \leq p < q \leq \infty$. Then $\BHD^q(H) \simeq \rr$ implies $\HD^p(G) \simeq \rr$. Also $\HD^p(H) \simeq \rr$ implies $\HD^p(G) \simeq \rr$.
\end{itemize}
\end{cor}
\begin{exa}
The classical example is to take $L$ finitely generated group (``lamp state'') and $C = \oplus_H L$ with $H$ acting by shifting the index. $C$ is the ``lamp configuration group'', and the semi-direct product is called a lamplighter group. If $H$ and $L$ are finitely generated, then $S_C$ can be picked to be the generating set of $L$ (at the index $e_H$). This groups satisfies \eqref{hTC} since subgroups which are ``far away'' will be in different factors of the direct product, and their elements commute.

Georgakopoulos \cite{Agelos} showed lamplighter \emph{graphs} do not have harmonic functions with gradient in $\ell^2$. His methods extends to harmonic functions with gradient in $\ell^p$. However, the lamp groups must be finite.

Using Theorem \ref{leteo}, \cite[\S{}4]{moi-trans}, \cite[Theorem.(iv)]{MV} and work of Georgakopoulos \cite{Agelos}, one may readily check that the only lamplighter groups for which it is not proven that $\HD^p(\Gamma) \simeq \rr$ for any $p \in [1,\infty[$ are those where $L$ is infinite amenable and either 1- $H$ is of polynomial growth and not virtually Abelian or 2- $H$ has $\HD^p(H) \not\simeq \rr$.
\end{exa}
Though the proof was not done in this generality, Theorem \ref{leteo} extends almost \emph{verbatim} to the case of lamplighter graphs. The correct hypothesis is that the graph $H$ must have $\mathrm{IS}_d$ for $d >2p$ and either: 1- $\HD^p(H) \simeq \rr$ or 2- $\BHD^q(H) \simeq \rr$ for some $q \in \big] \frac{dp}{d-2p}, \infty \big]$. This gives a partial answer to a problem raised by Georgakopoulos \cite[Problem 3.1]{Agelos}.
\begin{exa}\label{exsym}
\newcommand{\sym}[0]{\ensuremath{\mathrm{Sym}}}
Another classical example is to take $\sym_H$ to be the permutations $H \to H$ which are not the identity only on a finite set. There is a natural action (say, on the right) of $H$ on itself by permutation. This gives $G = \sym_H \rtimes_\phi H$ which is finitely generated (although $\sym_H$ is not finitely generated).
\end{exa}

\subsection{Further comments and questions}\label{s-comq}

A simple way to show that the gradient of a harmonic function is not in $\ell^p$ is to think in terms of electric currents. This is essentially the method used by Georgakopoulos \cite{Agelos} to show lamplighter \emph{graph} do not have harmonic functions with gradient in $\ell^2$. Indeed, if one exhibits ``many'' paths which are ``not too long'' between points where the potential is $\geq 5/8$ and points where its $\leq 3/8$, then one gets a lower bound on the gradient.

Note that, for groups, using \cite[Theorem 1.2]{moi-poiss}, it is sufficient to consider the case of bounded harmonic functions. Still, assume for simplicity that $f$ is a bounded harmonic function. Then, up to normalisation, its values are between $0$ and $1$. Let $N = f^{-1}[0,3/8]$ and $P=f^{-1}[5/8,1]$ (both are infinite sets). Let $k_n$ be the maximal number of edge-disjoint paths of length $\leq n$ between $N$ and $P$. Then the $\ell^p$ norm of the current is at least 
\[
\frac{1}{4} \sum_{\text{paths}} \frac{1}{\text{length of the paths}} \geq \frac{ k_n }{ n^p}
\]
Let $p_c := \inf \{ p \mid \HD^p(\Gamma) \not\simeq \rr \}$. Then, $\forall \eps>0$, $\limm{n \to \infty} \frac{k_n}{n^{p_c+\eps}} = 0$. 
This gives an amusing view of the critical exponent from $\ell^p$-cohomology, see Bourdon, Martin \& Valette \cite{BMV} or Bourdon \& Pajot \cite{BP}.
\begin{quest}
Given a group of exponential growth and divergence rate $n \mapsto n^d$, is it possible to show that $k_n$ grows exponentially?
\end{quest}
Indeed, there are always two geodesic rays $\{x_n\}$ and $\{y_n\}$ with $f(x_n) \to 1$ and $f(y_n) \to 0$. If the divergence does not grow too quickly, then there should be [exponentially] many paths of distance roughly $K n^{1/d}$ between those (for some $K>0$).


\begin{thebibliography}{10}

\bibitem{BMV}
M.~Bourdon and H.~Martin and A.~Valette,
\newblock Vanishing and non-Vanishing of the first $L_p$ cohomology of groups, 
\newblock \emph{Comment. Math. Helv.} \textbf{80}:377--389, 2005.

\bibitem{BP}
M.~Bourdon and H.~Pajot,
\newblock Cohomologie $\ell^p$ et espaces de Besov, 
\newblock \emph{J. reine angew. Math.}, \textbf{558}:85--108, 2003.

\bibitem{CM}
T.~Colding and W.~P.~Minicozzi,
\newblock Harmonic functions on manifolds,
\newblock \emph{Ann. of Math.} \textbf{146}(3):725--747, 1997.

\bibitem{Agelos}
A.~Georgakopoulos, 
\newblock Lamplighter graphs do not admit harmonic functions of finite energy,
\newblock \emph{Proc. Amer. Math. Soc.} \textbf{138}(9):3057--3061, 2010. 

\bibitem{moi-poiss}
A.~Gournay,
\newblock Boundary values of random walks and $\ell^p$-cohomology in degree one,
\newblock arXiv:1303.4091

\bibitem{moi-trans}
A.~Gournay,
\newblock Vanishing of $\ell^p$-cohomology and transportation cost,
\newblock \emph{Bull. London Math. Soc.} \textbf{46}(3):481--490, 2014.

\bibitem{moi-wreath}
A.~Gournay,
\newblock Harmonic functions with finite $p$-energy on lamplighter graphs are constant, 
\newblock arXiv:1502.02269

\bibitem{Gro}
M.~Gromov,
\newblock Asymptotic invariants of groups, \textit{in} \emph{Geometric group theory ({V}ol. 2)},
\newblock London Mathematical Society Lecture Note Series, {V}ol. \textbf{182}, Cambridge University Press, 1993, viii+295.


\bibitem{Kleiner}
B.~Kleiner,
\newblock A new proof of Gromov's theorem on groups of polynomial growth,
\newblock \emph{J. Amer. Math. Soc.} \textbf{23}(3):815--829, 2010.


\bibitem{MV}
F.~Martin and A.~Valette,
\newblock On the first $\mathsf{L}^p$ cohomology of discrete groups.
\newblock \emph{Groups Geom. Dyn.}, \textbf{1}:81--100, 2007.

\bibitem{MPTY}
T.~Meyerovitch, I.~Perl, M.~Tointon and A.~Yadin,
\newblock Polynomials and harmonic functions on discrete groups,
\newblock arXiv:1505.01175

\bibitem{Sogge}
C.~D.~Sogge,
\newblock \emph{Fourier integrals in classical analysis}, Cambridge tracts in mathematics, \textbf{105}. 
\newblock Cambridge University Press, 2008. 


\end{thebibliography}
\end{document}